\documentclass[11pt]{article}
\usepackage{amsmath,amssymb}

\newtheorem{propo}{{\bf Proposition}}[section]
\newtheorem{coro}[propo]{{\bf Corollary}}
\newtheorem{lemma}[propo]{{\bf Lemma}} \newtheorem{theor}[propo]{{\bf
Theorem}} \newtheorem{ex}{{\sc Example}}[section]
\newtheorem{definition}{\bf Definition}
\newenvironment{proof}{{\bf Proof.}}{$\Box$}

\def\N{{\mathbb N}}

\begin{document}

\vspace*{1.0in}

\begin{center} ALMOST-REDUCTIVE AND ALMOST-ALGEBRAIC LEIBNIZ ALGEBRAS
\end{center}
\bigskip

\begin{center} DAVID A. TOWERS 
\end{center}
\bigskip

\begin{center} Department of Mathematics and Statistics

Lancaster University

Lancaster LA1 4YF

England

d.towers@lancaster.ac.uk 
\end{center}
\bigskip

{\bf Abstract}

 This paper examines whether the concept of an almost-algebraic Lie algebra developed by Auslander and Brezin in \cite{ab} can be introduced for Leibniz algebras. Two possible analogues are considered: almost-reductive and almost-algebraic Leibniz algebras. For Lie algebras these two concepts are the same, but that is not the case for Leibniz algebras, the class of almost-algebraic Leibniz algebras strictly containing that of the almost-reductive ones. Various properties of these two classes of algebras are obtained, together with some relationships to $\phi$-free, elementary, $E$-algebras and $A$-algebras.
\medskip

\par


\bigskip

\section{Introduction}
\medskip

An algebra $L$ over a field $F$ is called a {\em Leibniz algebra} if, for every $x,y,z \in L$, we have
\[  [x,[y,z]]=[[x,y],z]-[[x,z],y]
\]
In other words the right multiplication operator $R_x : L \rightarrow L : y\mapsto [y,x]$ is a derivation of $L$. As a result such algebras are sometimes called {\it right} Leibniz algebras, and there is a corresponding notion of {\it left} Leibniz algebras, which satisfy
\[  [x,[y,z]]=[[x,y],z]+[y,[x,z]].
\]
Clearly the opposite of a right (left) Leibniz algebra is a left (right) Leibniz algebra, so, in most situations, it does not matter which definition we use. A {\em symmetric} Leibniz algebra $L$ is one which is both a right and left Leibniz algebra and in which $[[x,y],[x,y]]=0$ for all $x,y\in L$. This last identity is only needed in characteristic two, as it follows from the right and left Leibniz identities otherwise (see \cite[Lemma 1]{jp}). Symmetric Leibniz algebras $L$ are flexible, power associative and have $x^3=0$ for all $x\in L$ (see \cite[Proposition 2.37]{feld}), and so, in a sense, are not far removed from Lie algebras.
\par

Put $I=\langle\{x^2:x\in L\}\rangle$. Then $I$ is an ideal of $L$ and $L/I$ is a Lie algebra called the {\em liesation} of $L$. We define the following series:
\[ L^1=L,L^{k+1}=[L^k,L] \hbox{ and } L^{(1)}=L,L^{(k+1)}=[L^{(k)},L^{(k)}] \hbox{ for all } k=2,3, \ldots
\]
Then $L$ is {\em nilpotent} (resp. {\em solvable}) if $L^n=0$ (resp.$ L^{(n)}=0$) for some $n \in \N$. The {\em nilradical}, $N(L)$, (resp. {\em radical}, $\Gamma(L)$) is the largest nilpotent (resp. solvable) ideal of $L$.
\par

Throughout, $L$ will denote a (right) Leibniz algebra over a field $F$ of characteristic zero unless otherwise specified. The Frattini ideal of $L$, $\phi(L)$, is the largest ideal of $L$ contained in all maximal subalgebras of $L$. The Lie/Leibniz algebra $L$ is called {\em $\phi$-free} if $\phi(L) = 0$, and {\em elementary} if $\phi(B)=0$ for every subalgebra $B$ of $L$. Lie/Leibniz algebras all of whose nilpotent subalgebras are abelian are called {\em $A$-algebras}; Lie/Leibniz algebras $L$ such that $\phi(B)\leq \phi(L)$ for all subalgebras $B$ of $L$ are called {\em $E$-algebras}. The {\em abelian socle}, $Asoc(L)$, of a Lie/Leibniz algebra $L$ is the sum of its minimal abelian ideals.
\par

A linear Lie algebra $L\leq {\rm gl}(V)$ is {\it almost algebraic} if $L$ contains the nilpotent and semisimple Jordan components of its elements; an abstract Lie algebra $L$ is then called almost algebraic if ${\rm ad}L\leq {\rm gl}(L)$ is almost algebraic. Here we are exploring whether an analogous concept to this last one can be developed for Leibniz algebras and then to determine properties of, and inter-relationships between, these five classes of algebras analogous to those obtained by Towers and Varea in \cite{further}.  
\par

In section 2 the concepts of an almost-reductive Leibniz algebra and of an almost-algebraic Leibniz algebra are introduced and various basic properties of them are produced. Descriptions of symmetric Leibniz algebras wich are almost algebraic, and of those with an almost-reductive radical are obtained. In additionh, some analogues of the results in \cite{further} are found for symmetric Leibniz algebras. In section 3 the inner derivation algebra, $R(L)$, of $L$ is defined. This is a Lie algebra and some properties of almost-reductive and almost-algebraic Leibniz algebras $L$ are related to corresponding properties of $R(L)$. We also introduce the concept of an $L$-split element of a Leibniz algebra $L$ and show that $L$ is almost algebraic if every element of $L$ is $L$-split. It is also shown that if every element of a subalgebra $B$ of $L$ is $L$-split then the idealiser of $B$ in $L$ is almost algebraic. In the final section we determine some consequences of these results for symmetric Leibniz $A$-algebras.

\section{Definitions and Preliminary results}
\begin{definition}
We call $L$ {\bf almost reductive} if $L=N(L)\dot{+}\Sigma$, where $\Sigma$ is a Lie algebra and $N(L)$ is a completely reducible $\Sigma$-bimodule.
\end{definition}

\begin{lemma}\label{red} Let $L$ be an almost-reductive Leibniz algebra. Then $\Sigma=C\oplus S$, where $S$ is a semisimple Lie algebra, $C$ is an abelian Lie algebra and $R_c|_{N(L)}$ is semisimple for all $c\in C$.
\end{lemma}
\begin{proof} If $A$ is an irreducible $\Sigma$-bimodule of $L$ we have that $[\Sigma,A]=0$ or $[\sigma,a]=-[a,\sigma]$ for all $a\in A, \sigma \in \Sigma$, by \cite[Lemma 1.9]{barnes}. It follows that $A$ is an irreducible right $\Sigma$-module of $L$ and the result now follows from \cite[Theorem 11, p.47]{jac}.
\end{proof}

\begin{definition} We call $L$ {\bf almost algebraic} if $L/I$ is an almost-algebraic Lie algebra.
\end{definition}

\begin{theor}\label{t:aa} Let $L$ be an almost-reductive Leibniz algebra. Then $L$ is almost algebraic.
\end{theor}
\begin{proof} Let $L=N(L)\dot{+} \Sigma$ be an almost-reductive Leibniz algebra, where $\Sigma=C\oplus S$, and let $N/I$ be the nilradical of $L/I$. Then $N(L)\subseteq N$ and $N=N(L)+N\cap \Sigma$. Since $N\cap \Sigma$ is a solvable ideal of $\Sigma$, $N\cap \Sigma \subseteq C$. Let $C=N\cap \Sigma \oplus D$. Then $L/I=N/I\dot{+} (\Sigma'+I)/I$ where $\Sigma'=D\oplus S$ is reductive. Hence $L/I$ is an almost-algebraic Lie algebra.
\end{proof}
\medskip

The converse of the above result is false, as the following example shows.

\begin{ex}\label{ex1} Let $L$ be the four-dimensional solvable cyclic Leibniz algebra with basis $a,a^2,a^3,a^4$ and $[a^4,a]=a^4$. Then $I=L^2$ and $L/I$ is trivially almost algebraic. But $N(L)=I$ and $L$ is not completely reducible: for example, there is no ideal $A$ of $L$ such that $I=A\oplus (Fa^3+Fa^4)$.
\end{ex}

In fact, it even fails for symmetric Leibniz algebras, as shown in the next example.

\begin{ex}\label{ex2} Let $L$ be the three-dimensional symmetric Leibniz algebra with basis $e_1,e_2,e_3$ and non-zero products $[e_1,e_2]=e_1$, $[e_2,e_1]=-e_1$, $e_2^2=e_3$. Then $I=Fe_3$ and $L/I\cong Fe_1+Fe_2$, which has nilradical $Fe_1$ and is clearly almost algebraic. However, $N(L)=Fe_1+Fe_3$ and $L$ does not split over this ideal, so $L$ is not almost reductive.
\end{ex}

\begin{theor}\label{phifree} Let $L$ be a Leibniz algebra.
\begin{itemize} 
\item[(i)] If $L$ is $\phi$-free, then $L$ is
almost reductive (and so, almost algebraic). 
\item[(ii)] Let $L$ be almost reductive. Then $L$ is
$\phi$-free if and only if its nilradical is abelian.
\end{itemize} 
\end{theor} 
\begin{proof} (i) Let $L$ be
$\phi$-free. By \cite[Theorem 2.4 and Corollary 2.9]{stit} we have that $L = N(L) \dot{+} V$ where $V$
is a Lie subalgebra of $L$ acting completely reducibly on $N(L)$ and $N(L)=$ Asoc $L$. It follows that $L$ is almost reductive.

(ii) Suppose that $L$ is almost reductive and that $N(L)$ is abelian. Then $N(L) =$ Asoc $L$, so $L$ is $\phi$-free by the same argument as in \cite{frat}.
The converse follows from \cite[Theorem 2.4 and Corollary 2.9]{stit}.
\end{proof}

\begin{coro}\label{arphifree} Let $L$ be an almost reductive Leibniz algebra. Then $L$ is $\phi$-free if and only if $L=\Lambda \dot{+} I$, where $\Lambda=\Sigma\dot{+}A$  is a $\phi$-free almost-algebraic Lie algebra with nilradical $A$, $N(L)=A\oplus I$, $[I,\Sigma]=I$ and $I$ is a completely reducible $\Sigma$-bimodule.
\end{coro}
\begin{proof} First, let $L$ be $\phi$-free. Then $L= (A_1\oplus \ldots\oplus A_n)\dot{+}\Sigma$ where $\Sigma$ is as described in Lemma \ref{red} and $A_i$ is an abelian irreducible $\Sigma$-bimodule for each $1\leq i\leq n$. Suppose that $[a,\sigma]=-[\sigma,a]$ for all $a\in A_1\oplus \ldots\oplus A_r$, $\sigma\in \Sigma$, but that there is an $a_i\in A_i$ and a $\sigma_i \in \Sigma$ such that $[a_i,\sigma_i]\neq -[\sigma_i,a_i]$ for each $r+1\leq i\leq n$. Put $\Lambda=A_1\oplus \ldots\oplus A_r\dot{+}\Sigma$. Then $[A_i,\Sigma]$ is a $\Sigma$-bimodule, and so $[A_i,\Sigma]=A_i$ or $0$ for each $i$. Now $[\Sigma,A_i]=0$ for $r+1\leq i\leq n$, by \cite[Lemma 1.9]{barnes}, so $[A_i,\Sigma]=A_i$ for $r+1\leq i\leq n$, by the choice of $r$. If $a\in A_i$ and $\sigma\in \Sigma$, then $[a,\sigma]=[a,\sigma+[\sigma,a]\in I$ for $r+1\leq i\leq n$. It follows that $A_i=[A_i,\Sigma] \subseteq I$ for each $r+1\leq i\leq n$. Clearly, $A_{r+1}\oplus\ldots\oplus A_n\subseteq I$ since $\Lambda$ is a Lie algebra.
\par

Now let  $L=\Lambda \dot{+} I$, where $\Lambda=\Sigma\dot{+}A$  is a $\phi$-free almost-algebraic Lie algebra with nilradical $A$, $N(L)=A\oplus I$, $[I,\Sigma]=I$ and $I$ is a completely reducible $\Sigma$-bimodule. Then $L$ is $\phi$-free since its nilradical is abelian..
\end{proof}
\medskip

The following result is a generalisation of \cite[Theorem 6 and its Corollary]{ao}.

\begin{coro}\label{lie1} Every $\phi$-free Leibniz algebra in which $I\subseteq Z(L)$ is a Lie algebra; in particular, every $\phi$-free symmetric Leibniz algebra is a Lie algebra.
\end{coro}
\begin{proof} Let $L$ be a $\phi$-free Leibniz algebra in which $I\subseteq Z(L)$. It follows from Corollary \ref{arphifree} that $I=0$, and so $L$ is a Lie algebra. 
\end{proof}

\begin{definition} The {\bf right centre} of a Leibniz algebra is the set $Z_r(L)=\{z\in L \mid [x,z]=0 \hbox{ for all } x\in L\}$.
\end{definition}

For any (right) Leibniz algebra $L$, $Z_r(L)$ is an abelian ideal of $L$ (see \cite[Proposition 2.9]{feld}). A special case of the above result is the following.

\begin{propo}(cf. \cite[Theorem 6 and its Corollary]{ao}) If $L/Z_r(L)$ is semisimple and $\dim Z_r(L)=1$, then $L$ is a Lie algebra.
\end{propo}
\begin{proof} By Levi's Theorem for Leibniz algebras, $L=Z_r(L)\dot{+} S$ where $S$ is a semisimple Lie algebra. Let $Z_r(L)=Fz$ and let $s_1,s_2\in S$. Then $[z,s_i]=\lambda_i z$ for $i=1,2$ and $[z,[s_1,s_2]]=[[z,s_1],s_2]-[[z,s_2],s_1]=\lambda_1\lambda_2 z - \lambda_2\lambda_1 z=0$. Thus $[Z_r,L]=[Z_r,S]=[Z_r,S^2]=0$, whence $Z_r(L)=Z(L)$. It is now clear that $L$ is a Lie algebra. Note that $L$ is $\phi$-free and $I=0$, so this is a special case of Corollary \ref{lie1}.
\end{proof}

\begin{lemma}\label{frac} Let $B$ be a subalgebra of a Leibniz algebra $L$. If $B$ is almost algebraic, then so the Lie algebra $(B+I)/I$.
\end{lemma}
\begin{proof}  Let $J$ be the Leibniz kernel of $B$. Then $B/J$ is almost algebraic. But now
\[ \frac{B+I}{I} \cong \frac{B}{B\cap I} \cong \frac{B/J}{(B\cap I)/J},
\] and $(B\cap I)/J$ is abelian and so is almost algebraic. The result then follows from \cite[Lemma 4.1]{ab}.
\end{proof}

\begin{theor}\label{rad1} Let $L$ be a Leibniz algebra with radical $\Gamma$. Then,
\begin{itemize}
\item[(i)] if $\Gamma$ is almost algebraic, then so is $L$; and
\item[(ii)] if $L$ is almost reductive, then so is $\Gamma$.
\end{itemize}
\end{theor}
\begin{proof} (i) If $\Gamma$ is the radical of $L$, $\Gamma/I$ is the radical of $L/I$. Let $\Gamma$ be almost algebraic. Then $\Gamma/I$ is almost algebraic, by Lemma \ref{frac}. It follows from \cite[Corollary 3.1]{ab} that $L/I$, and hence $L$, is almost algebraic. 
\par

(ii) This is clear from Lemma \ref{red}.
\end{proof}
\medskip

For Lie algebras the converse is true. However, it appears that this may not be the case even for symmetric Leibniz algebras, though examples are not easy to construct. The best that we can achieve at the moment is given by the following two results.

\begin{theor}\label{symm1} Let $L$ be a almost-algebraic symmetric Leibniz algebra with radical $\Gamma$ and nilradical $N$. Then $L=N+\Sigma$, where $N\cap \Sigma=I$ and $\Sigma=\Gamma\cap\Sigma\oplus S$ with $(\Gamma\cap\Sigma)^3=0$, $S$ semisimple and $R_{c+I}$ acting semisimply on $N/I$ for all $c\in \Gamma\cap \Sigma$. 
\end{theor}
\begin{proof} We have that $L=\Gamma\dot{+}S$, where $S$ is a semisimple Lie algebra, by Levi's Theorem for Leibniz algebras. Also, $L/I$ is almost reductive. Thus, $L/I=N/I\dot{+}\Sigma/I$ where $\Sigma$ is a subalgebra of $L$, $(\Gamma\cap\Sigma)/I$ is abelian, $\Sigma/I=(\Gamma\cap\Sigma)/I\oplus (S\dot{+}I)/I$ with $S$ a semismple Lie algebra and $R_{c+I}$ acting semisimply on $N/I$ for all $c\in \Gamma\cap \Sigma$. Hence $L=N+\Sigma$ where $N\cap\Sigma=I$ and $(\Gamma\cap\Sigma)^2\subseteq I$, so $(\Gamma\cap\Sigma)^3=0$. Moreover,
\begin{align*}
[\Gamma\cap\Sigma,S]& = [\Gamma\cap\Sigma,S^2]\subseteq [[\Gamma\cap\Sigma,S],S]\subseteq [I,S]=0, \hbox{ and } \\
[S,\Gamma\cap\Sigma]& = [S^2,\Gamma\cap\Sigma]\subseteq [S,[S,\Gamma\cap\Sigma]]+[[S,\Gamma\cap\Sigma],S] \\
& \subseteq [S,I]+[I,S]=0.
\end{align*}
\end{proof}

\begin{theor}\label{symm2} Let $L$ be a symmetric Leibniz algebra with an almost-reductive radical $\Gamma$. Then $L$ is as described in Theorem \ref{symm1} above and $\Gamma=N\dot{+}C$ where $C$ is an abelian subalgebra and $R_c\mid_N$ is semisimple for all $c\in C$. 
\end{theor}
\begin{proof} Suppose that $\Gamma$ is almost reductive, so that $\Gamma=N\dot{+}C$ where $C$ is an abelian subalgebra and $R_c\mid_N$ is semisimple for all $c\in C$. Moreover, $\Gamma$ is almost-algebraic, by Theorem \ref{t:aa}, and hence so is $L$, by Theorem \ref{rad1}(i). The result follows.
\end{proof}

\begin{propo}\label{p:fac} Let $L$ be a Leibniz algebra $L$. 
\begin{itemize}
\item[(i)] If $L$ is almost algebraic and $J$ is an almost-algebraic ideal of $L$, then $L/J$ is almost algebraic.
\item[(ii)] If $L$ is almost reductive and $J$ is an ideal of $L$ with $J\subseteq \phi(L)$, then $L/J$ is almost reductive.
\end{itemize}
\end{propo}
\begin{proof} (i) It follows from Lemma \ref{frac} and \cite[Lemma 4.1]{ab} that $(J+I)/I$ and hence $L/(J+I)\cong (L/I)/((J+I)/I)$ is an almost-algebraic Lie algebra. But $(J+I)/J$ is the Leibniz kernel of $L/J$ and $(L/J)/((J+I)/J)\cong L/(J+I)$. Thus $L/J$ is almost algebraic.
\par

\noindent (ii) We have that $N(L/J)=N(L)/J$ as in \cite[Lemma 2.3]{nil}. Let $N(L)=A_1\oplus \ldots \oplus A_n$ where $A_i$ is an irreducible $\Sigma$-bimodule of $L$. Then $A_i\cap J= 0$ or $A_i$ for each $i=1,\ldots,n$. Let $A_1\cap J=\ldots=A_r\cap J=0$, $J=A_{r+1}\oplus \ldots \oplus A_n$, so $L/J\cong (A_1\oplus \ldots \oplus A_r)\dot{+} \Sigma$, which is almost reductive. 
\end{proof}

\begin{coro}\label{c:aalg} Let $L$ be an almost-reductive Leibniz algebra. Then $\phi(L) = N^2$, where $N$ is the nilradical of $L$.
\end{coro}
\begin{proof} First, $N^2=\phi(N)\subseteq \phi(L)$, as in \cite[Theorem 6.5]{frat}. Hence $N(L/N^2)=N/N^2$, by \cite[Lemma 2.3]{nil}, giving that  $N(L/N^2)$ is abelian. Moreover, $L/N^2$ is almost reductive, by Proposition \ref{p:fac} (ii), and so $L/N^2$ is $\phi$-free, by Theorem \ref{phifree} (ii). It
follows that $\phi(L) \subseteq N^2$. 
\end{proof}
\bigskip

Note that the above Corollary is false if 'almost-reductive' is replaced by `almost-algebraic', as the following example shows.

\begin{ex} Let $L$ be as in Example \ref{ex1}. Then the only maximal subalgebras are $I$ and $F(a-a^2)+F(a^2-a^3)+F(a^3-a^4)$ (see \cite[proof of Proposition 6.1]{st}. Hence $\phi(L)= F(a^2-a^3)+F(a^3-a^4)\neq 0=N^2$.
\end{ex}

Once again, it is not even true if $L$ is a symmetric Leibniz algebra.

\begin{ex} Let $L$ be as in Example \ref{ex2}. Then $\phi(L)=Fe_3\neq 0 = N^2$.
\end{ex}

\begin{propo}  Let $L$ be an almost-reductive symmetric Leibniz algebra. If every almost-algebraic subalgebra of $L$ is $\phi$-free, then $L$ is an elementary Lie algebra.
\end{propo}
\begin{proof} Since $\phi(L)=0$, we have that $L$ is a Lie algebra, by Corollary \ref{lie1}. The result now follows from \cite[Proposition 2.3]{further}.
\end{proof}

\begin{propo}\label{elem}  Let $L$ be an almost-reductive symmetric Leibniz algebra. If every almost-algebraic subalgebra of $L/I$ is $\phi$-free, then $\phi(L)=N^2=I$ and $L$ is an $E$-algebra.
\end{propo}
\begin{proof} We have that $L/I$ is almost algebraic, by Theorem \ref{t:aa}, so $N^2=\phi(L)\subseteq I$, by Corollary \ref{c:aalg}. Now, if $M$ is a maximal subalgebra of $L$ with $Z(L)\not \subseteq M$, $L=M+Z(L)$ which gives that $L^2\subseteq M$. Hence $Z(L)\cap L^2\subseteq \phi(L)$. Thus 
\[ N^2\subseteq I\subseteq Z(L)\cap L^2\subseteq \phi(L)=N^2.
\]
 Let $B$ be a subalgebra of $L$. Then $I_{L/I}((B+I)/I)$ is almost algebraic, by \cite[Theorem 2.3]{ab}, and so is $\phi$-free, by assumption. It follows that $(B+I)/I$ is $\phi$-free, by \cite[Lemma 4.1]{frat}, whence, $\phi(B)\subseteq I=\phi(L)$ and $L$ is an $E$-algebra.
\end{proof}

\section{The inner derivation algebra of a Leibniz algebra}

\begin{definition} The {\bf inner derivation algebra} of $L$ is the set $R(L)=\{R_x\mid x\in L\}$. 
\end{definition}

Note that $R(L)$ is a Lie algebra under bracket product.  For every subset $U$ of $L$ we will write $R_U=\{R_x\mid x\in U\}$. It is easy to check that $R_{[y,x]}=[R_x,R_y]$. To simplify notation, put $[y,_nx]=R_x^n(y)$. Then we have

\begin{lemma}\label{1} $R_{[y,_nx]}=(-1)^n [R_y,_{n-1}R_x]$.
\end{lemma}
\begin{proof} We use induction on n. If $n=1$ we have  $$R_{[y,x]}=[R_x,R_y]=-[R_y,R_x].$$
\par

So suppose it holds for $n=k$. Then
\begin{align*}
 R_{[y,_{k+1}x]}=R_{[[y,_k x],x]}=-[R_{[y,_k x]},R_x]& =(-1)^{k+1}[[R_y,_{k-1} R_x],R_x] \\
& =(-1)^{k+1}[R_y,_k R_x].
\end{align*}
\end{proof}

\begin{definition} For any algebra $A$, the {\bf opposite algebra}, $A^{\circ}$, has the same underlying vector space and the opposite multiplication, $(x,y)\mapsto x\star y =yx$, where juxtaposition denotes the multiplication in $A$.
\end{definition}

The following is easy to check (see, for example, \cite[Proposition 2.26]{feld}).

\begin{propo}\label{factor} For any (right) Leibniz algebra $L$, the map $\theta : L \rightarrow R(L)^{\circ} : x\mapsto R_x$ is a homomorphism with kernel $Z_r(L)$, so the Lie algebra $L/Z_r(L)$ is isomorphic to $R(L)^{\circ}$.
\end{propo}

The following lemma is easy to see.

\begin{lemma}\label{op} For any Lie algebra $L$,
\begin{itemize}
\item[(i)] $U$ is a subalgebra of $L$  if and only if $U^{\circ}$ is a subalgebra of $L^{\circ}$;
\item[(ii)]  $U$ is an ideal of $L$  if and only if $U^{\circ}$ is an ideal of $L^{\circ}$;
\item[(iii)] $U$ is solvable if and only if $U^{\circ}$ is solvable; 
\item[(iv)] $U$ is nilpotent if and only if $U^{\circ}$ is nilpotent; and
\end{itemize}
\end{lemma}

\begin{lemma}\label{2} For every Leibniz algebra $L$ we have
\begin{itemize}
\item[(i)] If $U$ is a subalgebra of $L$ then $R_U$ is a subalgebra of $R(L)$;
\item[(ii)] Every subalgebra of $R(L)$ is of the fom $R_U$ where $U$ is a subalgebra of $L$.
\item[(iii)] If $U$ is an ideal of $L$ then $R_U$ is an ideal of $R(L)$.
\item[(iv)]  Every ideal of $R(L)$ is of the fom $R_U$ where $U$ is an ideal of $L$.
\end{itemize}
\end{lemma}
\begin{proof}\vspace{-.3cm}
\begin{align*}
\hbox{(i) } U \hbox{ is a subalgebra of } L & \Rightarrow \lambda x + \mu y, [x,y]\in U \hbox{ for all } x,y\in U, \lambda,\mu\in F \\
& \Rightarrow \lambda R_x + \mu R_y = R_{\lambda x + \mu y}, [R_x,R_y]=R_{[y,x]}\in R_U.
\end{align*}
\noindent (ii) Let $K$ be a subalgebra of $R(L)$. Put $U=\{x\in L \mid R_x\in K\}$. Then $U$ is a subalgebra of $L$ as in (i).
\begin{align*}
\hbox{(iii) } U \hbox{ is an ideal of } L & \Rightarrow [x,y], [y,x]\in U \hbox{ for all } x\in U, y\in L \hspace{2.4cm} \\
& \Rightarrow [R_y,R_x], [R_x,R_y]=\pm R_{[x,y]}\in R_U.
\end{align*}
\noindent (iv) This is similar to (ii).
\end{proof}

\begin{lemma}\label{rad} Let $L$ be a Leibniz algebra. Then
\begin{itemize}
\item[(i)] $\Gamma$ is the radical of $L \Leftrightarrow R_{\Gamma}$ is the radical of $R(L)$;
\item[(ii)] If $Z_r(L)\subseteq \phi(L)$, then $N$ is the nilradical of $L\Leftrightarrow R_N$ is the nilradical of $R(L)$.
\end{itemize}
\end{lemma}
\begin{proof} (i) Clearly, $\Gamma(L/Z_r(L))=\Gamma/Z_r(L)$ and so $\Gamma/Z_r(L)\cong\Gamma(R(L)^{\circ})=\Gamma(R(L))$. Moreover, $\theta\mid_{\Gamma}$ is a homomorphism from $\Gamma$ onto $R_{\Gamma}$, whence the result.
\par

\noindent (ii) Clearly, $Z_r(L)\subseteq N$ and so $N/Z_r(L)\subseteq N(L/Z_r(L))=K/Z_r(L)$, say. But $K$ is nilpotent, by \cite[Theorem 5.5]{barnes}, so $N/Z_r(L)=N(L/Z_r(L))$. The proof now follows in similar manner to (i).
\end{proof}

\begin{propo}\label{main}If the Leibniz algebra $L$ is  almost algebraic then so is the Lie algebra $R(L)$.
\end{propo}
\begin{proof} Let $L$ be almost algebraic. Then $L/I$ is almost algebraic. Now $L/Z_r(L)\cong (L/I)/(Z_r(L)/I)$ and $Z_r(L)/I$ is abelian and so is almost algebraic. It follows that $L/Z_r(L)$ is almost algebraic, by \cite[Lemma 4.1]{ab}. The result follows from Proposition \ref{factor}.
\end{proof}

\begin{definition}
If $B$ is a subalgebra of $L$, the {\bf idealiser of $B$ in $L$}, $I_L(B)=\{ x\in L \mid [x,b], [b,x]\in B \hbox{ for all } b\in B\}$.
\end{definition}

\begin{coro}\label{idealiser2}  Let $B$ be a subalgebra of an almost-algebraic Leibniz algebra $L$. Then the idealiser, $I_{R(L)}(R_B)$, of $R_B$ in $R(L)$ is an almost-algebraic Lie algebra.
\end{coro}
\begin{proof} This follows from Proposition \ref{main} and \cite[Theorem 2.3]{ab}.
\end{proof}

\begin{definition} The element $x\in L$ is called {\bf $L$-split} if there exist elements $s,n\in L$ such that $R_x=R_s+R_n$ is the decomposition of $R_x$ into its semisimple and nilpotent parts.
\end{definition}

\begin{propo}\label{split} If every element of the Leibniz algebra $L$ is $L$-split, then $L$ is almost algebraic. 
\end{propo}
\begin{proof} Let $x\in L$. Then $R_{x+I}=R_{s+I}+R_{n+I}$ if $ R_x=R_s+R_n$, and $R_{s+I}, R_{n+I}$ are the semisimple and nilpotent parts of $R_{x+I}$, so the result follows from \cite[Theorem 2]{ab}. 
\end{proof}
\medskip

The following result is now proved as in \cite[Theorem 2.3]{ab}.

\begin{propo}\label{idealiser} Let $B$ be a subalgebra of a Leibniz algebra $L$ in which every element is $L$-split. Then the idealiser, $I_L(B)$, of $B$ in $L$ is almost algebraic.
\end{propo}
\begin{proof} Let $J=I_L(B)$. Since $R_L(B)$ leaves $B$ invariant, so does its algebraic hull. In particular, if $x\in J$, both the semisimple and nilpotent parts of $R_L(x)$ leave $B$ invariant. Hence,every element of $J$ is $J$-split, and so, by Proposition \ref{split}, $J$ is almost algebraic.
\end{proof}

\section{Leibniz $A$-algebras}

\begin{propo}\label{lie} Let $L$ be a Lie $A$-algebra and let  $K$ be an ideal of $L$ with $K\subseteq Z(L)$. If $L/K$ is almost algebraic then so is $L$.
\end{propo}
\begin{proof} Let $L/K$ be almost algebraic and let $R$ be the radical of $L$. Then $R/K$ is almost algebraic, by \cite[Corollary 3.1]{ab}, and so $\phi(R/K)=0$, by \cite[Lemma 2.1 (ii)]{further}. Hence $\phi(R)\subseteq K\subseteq Z(R)$, by \cite[Corollary 4.4]{frat}. It follows that $\phi(R)\subseteq Z(R)\cap R^2=0$, by \cite[Theorem 3.3]{lie}, since $R$ is an $A$-algebra. Thus $R$ is almost algebraic, by \cite[Proposition 2.1]{further}, whence so is $L$, by \cite[Corollary 3.1]{ab} again.
\end{proof}

\begin{coro}\label{leib} Let $L$ be a Leibniz $A$-algebra and let  $K$ be an ideal of $L$ with $K\subseteq Z(L)$. If $L/K$ is almost algebraic then so is $L$.
\end{coro}
\begin{proof} Let $L/K$ be almost algebraic. Then 
\[ \frac{L/I}{(I+K)/I}\cong \frac{L/K}{(I+K)/K},
\] which is almost algebraic, by Proposition \ref{p:fac}. Moreover, $(I+K)/I\subseteq Z(L/I)$ and $L/I$ is a Lie $A$-algebra, by \cite[Lemma 2]{lie}, so $L/I$ is almost algebraic, by Proposition \ref{lie}. Hence $L$ is almost algebraic.
\end{proof}

\begin{lemma}\label{ars} Let $L$ be an almost-reductive symmetric Leibniz $A$-algebra. Then $L$ is a Lie algebra.
\end{lemma}
\begin{proof} Since $L$ is almost reductive, $\phi(L)=N^2=0$, since $L$ is an $A$-algebra. The result now follows from Corollary \ref{lie1}.
\end{proof}

\begin{lemma}\label{R(L)} If $L$ is a symmetric Leibniz $A$-algebra, then $L/I$ is a Lie $A$-algebra.
\end{lemma}
\begin{proof} If $K/I$ is a nilpotent subalgebra of $L/I$, $K^r\subseteq I$ for some $r>0$, whence $K^{r+1}=0$. It follows that $K$ is nilpotent and thus abelian.   
\end{proof}

\begin{theor}\label{A} Let $L$ be a symmetric Leibniz $A$-algebra. Then $L$ is an almost-reductive algebra if and only if it is an elementary Lie algebra.
\end{theor}
\begin{proof} $(\Rightarrow)$ Let $L$ be an almost-reductive symmetric Leibniz $A$-algebra. Then $L$ is a Lie algebra, by Lemma \ref{ars}.It now follows that it is elementary, by \cite[Theorem 2.4]{further}.
\par

\noindent $(\Leftarrow)$  The converse follows from \cite[Theorem 2.4]{further}.
\end{proof}

\begin{coro} Let $L$ be a symmetric Leibniz $A$-algebra with radical $\Gamma$. If $\Gamma$ is $\phi$-free, then $L$ is an elementary Lie algebra.
\end{coro}
\begin{proof} Assume that $\Gamma$ is $\phi$-free. Then $\Gamma$ is almost reductive, by Corollary \ref{phifree} (i). It follows that $L$ is as described in Theorem \ref{symm2}. Moreover, $(\Gamma\cap \Sigma)^2=0$, since $L$ is an $A$-algebra, and, if $\sigma\in \Gamma\cap \Sigma$, $\sigma=n+c$ for some $n\in N$, $c\in C$. Hence $[n',\sigma]=[n',c]$ for all $n'\in N$, and so $R_{\sigma}\mid_N$ is semisimple. It follows that $L$ is almost reductive and hence that $L$ is an elementary Lie algebra, by Theorem \ref{A}.
\end{proof}

\begin{coro} Let $L$ be an almost-reductive symmetric Leibniz $A$-algebra. Then $L$ splits over each of its ideals.
\end{coro}
\begin{proof} This follows from Lemma \ref{ars} and \cite[Corollary 2.6]{further}.
\end{proof}

\begin{propo}\label{E} Let $L$ be a Leibniz algebra over any field. Then $L$ is an $E$-algebra if and only if $L/\phi(L)$ is elementary.
\end{propo}
\begin{proof} The proof is the same as for the Lie case in \cite[Proposition 2]{ernie}.
\end{proof}

\begin{propo} Let $L$ be a symmetric Leibniz $A$-algebra. Then $L$ is an $E$-algebra.
\end{propo}
\begin{proof} Let $L$ be a Leibniz $A$-algebra. Then $L/\phi(L)$ is an $A$-algebra , by \cite[Lemma 2]{Aalg}. But $L/\phi(L)$ is $\phi$-free and so is almost-reductive, by Corollary \ref{c:aalg} (i). Hence $L/\phi(L)$ is elementary, by Theorem \ref{A}, and so $L$ is an $E$-algebra, by Proposition \ref{E}.
\end{proof}

\end{document}